\documentclass{mjm}

\usepackage[all]{xy}
\usepackage{latexsym}
\usepackage{amssymb}
\usepackage{fullpage}
\usepackage{amsmath}
\usepackage{amsthm}
\usepackage{mjm}
\usepackage[dvips]{graphicx}
\usepackage{amsmath,amsfonts,mathrsfs,latexsym,paralist, xypic}

\usepackage{tikz}

\newtheorem{theorem}{Theorem}[subsection]

\newtheorem{corollary}[theorem]{Corollary}
\newtheorem{lemma}[theorem]{Lemma}

\theoremstyle{definition}
\newtheorem{remark}[theorem]{Remark}
\newtheorem{definition}[theorem]{Definition}

\newcommand{\la}{\langle}
\newcommand{\ra}{\rangle}

\newcommand{\V}{\Vert}

\newcommand{\Z}{\mathbb{Z}}

\newcommand{\T}{\mathbb{T}}

\newcommand{\C}{\mathbb{C}}

\newcommand{\inv}{^{-1}}

\numberwithin{equation}{section}

\title{Irreducible representations of nilpotent groups generate classifiable C*-algebras}
\headtitle{Irreducible representations of nilpotent groups generate classifiable C*-algebras}
\author{Caleb Eckhardt and Elizabeth Gillaspy}
\headauthor{Caleb Eckhardt and Elizabeth Gillaspy}
\address{Department of Mathematics, Miami University, Oxford, Ohio}
\email{eckharc@miamioh.edu}
\address{Department of Mathematics, University of Colorado - Boulder, Boulder, Colorado}
\email{elizabeth.gillaspy@colorado.edu}

\thanks{C.E.\ was partially supported by a grant from the Simons Foundation.}
\begin{document}
\begin{abstract} We show that C*-algebras generated by irreducible representations of finitely generated nilpotent groups satisfy the universal coefficient theorem of Rosenberg and Schochet.  This result combines with previous work to show that these algebras are classifiable by their Elliott invariants within the class of unital, simple, separable, nuclear C*-algebras with finite nuclear dimension that satisfy the universal coefficient theorem. We also show that these C*-algebras are central cutdowns of twisted group C*-algebras with homotopically trivial cocycles.
\end{abstract}
\maketitle

\section{Introduction} 
This note concludes a long line of study into the $C^*$-algebras generated by irreducible representations of finitely generated nilpotent groups. Specifically, we prove that  such $C^*$-algebras satisfy the universal coefficient theorem (UCT) of Rosenberg and Schochet \cite{Rosenberg87}.  We combine this with a slew of other results to show that these algebras
 are classifiable by their Elliott invariant within the class $\mathcal{C}$ of unital, simple, separable, nuclear C*-algebras with finite nuclear dimension that satisfy the UCT. 
 
Let $G$ be a finitely generated nilpotent group and $\pi$ an irreducible unitary representation of $G.$ Let $C^*_\pi(G)$ be the $C^*$-algebra generated by $\pi(G).$ Tikuisis, White and Winter recently took the final step in a long and beautiful journey of showing that two elements of $\mathcal{C}$ with the same Elliott invariant are isomorphic \cite[Corollary D]{Tikuisis15}.  Therefore our job consists of showing that $C^*_\pi(G)\in \mathcal{C}.$ 

We have known for a while that $C^*_\pi(G)$ is nuclear \cite{Lance73} and simple \cite{Moore76}. Recently, the first author together with McKenney \cite{Eckhardt14b} showed that $C^*_\pi(G)$ has finite nuclear dimension (this work directly relied on a long list of results including \cite{Eckhardt14,Matui12, Matui14a, Rordam04,Winter12} -- see the introduction of \cite{Eckhardt14b} for the full story). As pointed out in \cite[Theorem 4.5]{Eckhardt14b} the only missing ingredient to show $C^*_\pi(G)\in \mathcal{C}$ was the UCT.  It was also pointed out in \cite[Theorem 4.6]{Eckhardt14b} that if $G$ is torsion free and $\pi$ is a faithful representation then $C^*_\pi(G)$ satisfies the UCT.  This observation  has already found success in \cite{Eckhardt14a} where the authors calculated the Elliott invariant of C*-algebras generated by faithful irreducible representations of the (torsion free) unitriangular group $UT(4,\Z)$, thus classifying them within $\mathcal{C}.$

Proving Theorem \ref{thm:main}, that $C^*_\pi(G)$ satisfies the UCT, is the main goal of this note.  In the course of our investigations  we noticed that $C^*_\pi(G)$ is isomorphic to a central cutdown of a twisted group C*-algebra  $C^*(G/N,\sigma)$ for some 2-cocycle $\sigma$, where $N$ is a finite index subgroup of $Z(G).$ Moreover, the cocycle $\sigma$ is homotopic to the trivial cocycle.
We include this as Theorem  \ref{thm:structure} as it may be of independent interest and useful for K-theory calculations.
\section{Nilpotent lemmas} \label{sec:nilemmas}
We first recall and prove necessary facts about nilpotent groups used in subsequent sections. We refer the reader to Segal's book \cite{Segal83} for information about nilpotent and polycyclic groups.
Throughout this section $G$ is a \textbf{finitely generated nilpotent group.} 
  We let $Z(G)$ denote the center of $G$ and 
\begin{equation*}
G_f=\{ x\in G: \textup{the conjugacy class of }x\textup{ is finite } \}
\end{equation*}
Clearly $Z(G)\leq G_f  \leq G.$ Define the \emph{torsion subgroup } of $G$ as
\begin{equation*}
T(G)=\{ x\in G: x\textup{ has finite order } \}.
\end{equation*}
\begin{remark} In general, $T(G)$ need not be a subgroup of $G$ but it is a standard exercise to show that $T(G) \leq G$ for nilpotent $G$ (see \cite[Corollary 1.B.10]{Segal83}).
\end{remark}
Every finitely generated nilpotent group is \emph{polycyclic}, that is, there is a normal series
\begin{equation} \label{eq:poly}
\{ e \}=G_n \trianglelefteq G_{n-1}\trianglelefteq \cdots \trianglelefteq G_1\trianglelefteq G_0=G
\end{equation}
where $G_{i}/G_{i+1}$ is cyclic for each $i=0,\ldots,n-1.$ From this definition it follows that polycyclic groups are finitely generated and that subgroups of polycyclic groups are polycyclic. Moreover, for finitely generated nilpotent groups, $T(G)$ is polycyclic and therefore finite, since it must satisfy (\ref{eq:poly}). From this it easily follows that $T(G)\leq G_f.$
\\\\
It is most likely well-known among group theorists that $Z(G)$ has finite index in $G_f$ and that $G/G_f$ is torsion free.  We were not able to locate references for these facts (although they are more-or-less corollaries of Baer's \cite{Baer48}) so we include the brief proofs.
\begin{lemma} \label{lem:tfreecase} Let $G$ be torsion free, finitely generated and nilpotent.  Then $Z(G)=G_f.$
\end{lemma}
\begin{proof} By  \cite[Lemma 3]{Baer48}, the group  $Z(G_f)$ has finite index in $G_f.$  
\\\\
\emph{Case 1.} $Z(G_f)=Z(G).$  By \cite{Malcev49},  $G/Z(G)$ is torsion free. Then $G_f/Z(G_f)$ is a finite, torsion free group, i.e. $G_f=Z(G_f).$
\\\\
\emph{Case 2.} $Z(G_f)\neq Z(G).$   By assumption there are elements $x\in Z(G_f)\setminus Z(G)$ and $y\in G$ such that $yxy^{-1}\neq x.$  Let $\alpha$ be the automorphism of $Z(G_f)$ induced by conjugation by $y.$  Because $x\in G_f$ we  have $\alpha^n(x)=x$ for some $n>1.$    Moreover, $Z(G_f)$ is a finitely generated torsion free abelian group, since $G$ is nilpotent, torsion free, and finitely generated.
Since the group generated by $Z(G_f)$ and $y$ is nilpotent it follows that 
 the matrix for $\alpha$  is unipotent.  In particular, $1$ is the only eigenvalue of $\alpha$, so $\alpha$ cannot have any periodic, non-fixed points, a contradiction  to the fact that $\alpha^n(x)=x$.  Therefore $Z(G)=Z(G_f)$,  so Case 2 never occurs and we are in Case 1.
\end{proof}
\begin{lemma} \label{lem:findex} Let $G$ be a finitely generated nilpotent group, then $Z(G)$ has finite index in $G_f.$
\end{lemma}
\begin{proof}    Let $x\in G_f.$  By \cite[Lemma 3]{Baer48},   there is some $n\geq1$ such that $z=x^n \in Z(G_f).$ Let $y_1, \ldots, y_k$ generate $G.$ Since $zT(G)\in (G/T(G))_f$ and $G/T(G)$ is torsion free, by Lemma \ref{lem:tfreecase} it follows that $zT(G)\in Z(G/T(G))$, i.e. 
\begin{equation*}
[z,y_i]\in T(G) \textup{ for }i=1, \ldots, k.
\end{equation*}
Then for $n\geq 1$ we have
\begin{equation} \label{eq:commid}
[z^{n+1},y_i]=z^n[z,y_i]z^{-n}[z^n,y_i]=[z,y_i][z^n,y_i].
\end{equation}
The first equality holds in every group and the second follows because $z\in Z(G_f)$ and $[z,y_i]\in T(G)\leq G_f.$ 
By induction, one uses (\ref{eq:commid}) to show that $[z^n,y_i]=[z,y_i]^n$ for all $1\leq i\leq k$ and $n\geq1.$

Since $[z,y_i]\in T(G)$ it follows that there is some power of $z$ such that $[z^d,y_i]=e$ for all $i=1, \ldots, k.$  In particular $x^{nd}=z^d\in Z(G).$ Therefore $G_f/Z(G)$ is a torsion group.  But $G_f/Z(G)$ is polycyclic and therefore finite.
\end{proof}
\begin{lemma} \label{lem:tfree} Let $G$ be a finitely generated nilpotent group. Then $G/G_f$ is torsion free.
\end{lemma}
\begin{proof}  Suppose $x\in G\setminus G_f$ and $x^n\in G_f$ for some $n.$ By Lemma \ref{lem:findex} we have $x^m\in Z(G)$ for some $m.$  Since $G/T(G)$ is torsion free and $x^mT(G)\in Z(G/T(G))$ we have $xT(G)\in Z(G/T(G)).$ This means that for every $y\in G$ we have $yxy^{-1}x^{-1}\in T(G)$; equivalently, $yxy^{-1}=xz$ for some $z\in T(G).$  Since $T(G)$ is finite, this means $x\in G_f$, a contradiction.
\end{proof}

\section{Main result} \label{sec:main}
\begin{definition}
For a group $G$ and normalized positive definite function $\phi:G\rightarrow \C$, let $(\pi_\phi,H_\phi)$ denote the GNS representation of $G$ associated with $\phi.$ Let $C^*_{\pi_\phi}(G)$ denote the $C^*$-algebra generated by $\pi_\phi(G).$ A \textbf{trace} on $G$ is a normalized, positive definite function that is constant on conjugacy classes.  Notice that a trace $\tau$ on $G$ canonically induces a tracial state on $C^*(G),$ which we will also denote by the same symbol $\tau$.
\end{definition}
The following lemma is well-known and may be found for example in \cite[Proposition 4.1.9]{Brown08}.
\begin{lemma} Let $A$ be a C*-algebra and $\phi$ a faithful state on $A.$  Let $\alpha$ be an automorphism of $A.$  Let $u\in A\rtimes_\alpha \Z$ be the unitary implementing $\alpha$. One extends $\phi$ to a faithful state on $A\rtimes_\alpha\Z$ by setting $\phi(xu^n)=0$ when $x\in A$ and $n\neq0.$ \label{lem:faithstates} 
\end{lemma}

\begin{lemma} \label{lem:crossedproducts} Let $N$ be a discrete group and  $\alpha$ an automorphism of $N.$  Set $G=N\rtimes_\alpha \Z.$ Let $u\in G$ be the element implementing $\alpha$ by conjugation.
  Let $\tau$ be a trace on $G$ such that $\tau(x)=0$ for all $x\in G\setminus N.$ Then
  \begin{equation*}
  C^*_{\pi_\tau}(G)\cong C^*_{\pi_\tau}(N)\rtimes_{\textup{Ad}\pi_\tau(u)} \Z
  \end{equation*}
\end{lemma}
\begin{proof} 
Let $\sigma:C^*_{\pi_\tau}(N)\rtimes_{\textup{Ad}\pi_\tau(u)}\Z\rightarrow C^*_{\pi_\tau}(G)$ denote the surjective *-homomorphism corresponding to the covariant representation which is the inclusion map on $C^*_{\pi_\tau}(N)$ and sends $n\in \Z$ to $\pi_\tau(u)^n.$
By Lemma \ref{lem:faithstates},  $\tau\circ \sigma$ defines a faithful tracial state on $C^*_{\pi_\tau}(N)\rtimes_{\textup{Ad}\pi_\tau(u)} \Z$.  Therefore  $\sigma$ is injective. 
\end{proof}

\begin{theorem} \label{thm:obs} Let $G$ be finitely generated and nilpotent.   Let $\tau$ be a trace on $G$ such that $\tau(x)=0$ for all $x\in G\setminus G_f.$    Then $C^*_{\pi_\tau}(G)$ is isomorphic to an iterated crossed product of $C^*_{\pi_\tau}(G_f)$ by $\Z$-actions, i.e.
\begin{equation}
C^*_{\pi_\tau}(G)\cong C^*_{\pi_\tau}(G_f)\rtimes \Z \rtimes \cdots \rtimes \Z ,\label{eq:4}
\end{equation}
and $C^*_{\pi_\tau}(G)$ satisfies the UCT.
\end{theorem}
\begin{proof}  By Lemma \ref{lem:tfree}, $G/G_f$ is torsion free. Since $G/G_f$ is finitely generated and nilpotent it is isomorphic to an iterated semi-direct product of $\Z$-actions. By repeatedly using the fact that the short exact sequence of groups  $0\rightarrow A \rightarrow B\rightarrow \Z\rightarrow 0$ always splits we have
\begin{equation*}
G\cong G_f\rtimes \Z \rtimes \cdots \rtimes\Z.
\end{equation*}
Now \eqref{eq:4} follows from repeated applications of Lemma \ref{lem:crossedproducts}. By Lemma \ref{lem:findex}, the group C*-algebra $C^*(G_f)$ is subhomogeneous, from which it follows that $C^*_{\pi_\tau}(G_f)$ is also subhomogeneous.  In particular, $C^*_{\pi_\tau}(G_f)$ is Type I, so it satisfies the UCT by \cite[Theorem 1.17]{Rosenberg87}.
Finally, \cite[Proposition 2.7, Theorem 4.1]{Rosenberg87} shows that the UCT is preserved by $\Z$-actions, implying that $C^*_{\pi_\tau}(G)$ satisfies the UCT.
\end{proof}

\begin{theorem} \label{thm:main} Let $G$ be a finitely generated nilpotent group and $\pi$ an irreducible representation of $G.$  Then $C^*_\pi(G)$ satisfies the UCT.
\end{theorem}
\begin{proof} Suppose first that $\pi$ is faithful on $G.$ Since $G$ is finitely generated, by \cite[Proposition 5.1]{Carey84} there is an extreme trace $\tau$ on $G$ such that $C^*_\pi(G)\cong C^*_{\pi_\tau}(G).$ Since $\tau$ is an extreme trace and $\pi$ is faithful on $G$, it follows from \cite[Theorem 4.5]{Carey84} that $\tau(g)=0$ if $g\not\in G_f.$ By Theorem \ref{thm:obs}, $C^*_{\pi_\tau}(G)$ satisfies the UCT.

If $\pi$ is not faithful, then we replace $G$ with $G/\textup{ker}(\pi)$ and $\pi$ with $\widetilde{\pi}(g\textup{ker}(\pi)):=\pi(g)$.  Then $\widetilde{\pi}$ is faithful and $C^*_{\widetilde{\pi}}(G /\text{ker}(\pi)) \cong C^*_\pi(G)$, so we simply apply the above proof to $(G/\text{ker}(\pi), \widetilde{\pi})$.
\end{proof}

\begin{corollary} Let $G$ be a finitely generated nilpotent group and $\pi$ an irreducible representation of $G.$ Then $C^*_\pi(G)$ is classified by its ordered K-theory within the class of simple, unital, nuclear C*-algebras with finite nuclear dimension that satisfy the universal coefficient theorem.
\end{corollary}
\begin{proof} This is Theorem \ref{thm:main} combined with \cite[Theorem 4.5]{Eckhardt14b}.
\end{proof}

\section{A structure theorem for $C^*_\pi(G)$}
In the course of proving Theorem  \ref{thm:main} we discovered a structure theorem (Theorem \ref{thm:structure}) for $C^*_\pi(G)$ that is most likely unknown.  We therefore include it as it may be of interest to those working with twisted group $C^*$-algebras and representation theory.

Let $G$ be a discrete amenable group and $N\leq Z(G).$  Let $\omega\in \widehat{N}$ (the dual group of $N$) and also denote by $\omega$ the trivial extension of $\omega$ to $G$ (i.e., $\omega(x)=0$ for $x\not\in N.$)  
Let $c:G/N\rightarrow G$ be a choice of coset representatives. For $t\in G$, let $t_\omega\in H_\omega$ denote the canonical image of $t$ associated with the GNS representation. 

Recall that, for any $s,t\in G,$ we have $\la t_\omega,s_\omega \ra=\omega(s^{-1}t).$ Hence if $s^{-1}t\not\in N$, the vectors $s_\omega$ and $t_\omega$ are orthogonal.  On the other hand, for $s\in G$ we have $\la s_\omega, c(sN)_\omega  \ra=\omega(c(sN)^{-1}s)\in \T$, from which it follows that 
\begin{equation}
s_\omega=\omega(c(sN)^{-1}s)c(sN)_\omega. \label{eq:proportional}
\end{equation}

We deduce that $\{ s_\omega: s\in c(G/N) \}$ is an orthonormal basis for $H_\omega.$ 
 Define $W:H_\omega\rightarrow \ell^2(G/N)$ by  $W(s_\omega)=\delta_{sN}$ for $s\in c(G/N).$   Then $W$ is unitary.  Moreover, by (\ref{eq:proportional}) we have that, for any $y\in G$ and $s\in c(G/N)$, 
 \begin{equation}
W\pi_\omega(y)W^*(\delta_{sN})=\omega(c(ysN)^{-1}yc(sN))\delta_{ysN}. \label{eq:twist}
\end{equation}
Now (as in \cite{Packer92}) define the 2-cocycle $\sigma:G/N\times G/N\rightarrow \T$ by
\begin{equation}
\sigma(xN,yN)=\omega(c(xN)c(yN)c(xyN)\inv). \label{eq:cocycle}
\end{equation}
By an obvious adaptation of the proof of \cite[Proposition 2.9]{Eckhardt14a} and the discussion preceding it,
we have the following
\begin{lemma} \label{lem:twistedgroup} Let $G$ be a discrete amenable group and $N\leq Z(G).$ Let $\omega\in \widehat{N}.$  Also let $\omega$ denote the trivial extension to $G$ and define $\sigma$ as in (\ref{eq:cocycle}).  Then $C^*_{\pi_\omega}(G)$ is isomorphic to the twisted group $C^*$-algebra $C^*(G/N,\sigma).$  
\end{lemma}

Let $\lambda_{G/N}$ denote the left regular representation of $G$ on $\ell^2(G/N).$
\begin{lemma} Let $G$ be a discrete amenable group and $N\leq Z(G).$  Let $\omega\in \widehat{N}$ and denote by $\omega$ the trivial extension of $\omega$ to $G.$ Let $\tau$ be a trace on $G$ such that $\tau|_N=\omega|_N.$
Then $\lambda_{G/N}\otimes \pi_\tau$ is unitarily equivalent to $\pi_\omega\otimes 1_{H_\tau}.$ 

Writing $\prec$ to denote weak containment, it then follows that $\pi_\tau \prec \pi_\omega$; equivalently, $C^*_{\pi_\tau}(G)$ is a quotient of $C^*_{\pi_\omega}(G).$ \label{lem:weakcontainment}
\end{lemma}
\begin{proof} This is a slight modification of the proof of Fell's absorption principle.  Define a unitary $U$ on $\ell^2(G/N)\otimes H_\tau$ by $U(\delta_{tN}\otimes \xi)=\delta_{tN}\otimes \pi_\tau(c(tN))\xi.$ For any $y,t\in G$,
\begin{align*}
U^*(\lambda_{G/N}(y)\otimes \pi_\tau(y))U(\delta_{tN}\otimes \xi)&=U^*(\lambda_{G/N}(y)\otimes \pi_\tau(y))(\delta_{tN}\otimes \pi_\tau(c(tN))\xi)\\
&=U^*(\delta_{ytN}\otimes \pi_\tau(yc(tN))\xi)\\
&= \delta_{ytN}\otimes \pi_\tau(c(ytN)^{-1}yc(tN))\xi\\
&=\omega(c(ytN)^{-1}yc(tN))\delta_{ytN}\otimes\xi,
\end{align*}
because $\tau=\omega$ on $N$, both are multiplicative on $N$, and $c(ytN)^{-1}yc(tN)\in N.$ Unitary equivalence now follows from (\ref{eq:twist}). 
\\\\
We now show weak containment. Let $1_G$ denote the trivial representation of $G$  on $\C$. Since $G/N$ is amenable,  $\lambda_{G/N}$ contains an approximately fixed vector, and thus $1_G\prec \lambda_{G/N}.$  Then
\begin{equation*}
\pi_\tau\sim 1_G\otimes \pi_\tau\prec \lambda_{G/N}\otimes \pi_\tau\sim \pi_\omega\otimes 1_{H_\tau}\prec \pi_\omega. 
\end{equation*}\end{proof}
The following is well-known (see for example \cite[Corollary 2.5.12]{Brown08}).
\begin{lemma} \label{lem:condexp} Let $H\leq G$ be discrete groups. Let $C^*_r(G)$ denote the reduced group C*-algebra of $G.$ The  linear map from $\C[G]$ to $\C[H]$ defined by $E(\lambda_s)=\lambda_s$ if $s\in H$ and $E(\lambda_s)=0$ if $s\not\in H$ extends to a conditional expectation from $C^*_r(G)$ onto $C^*_r(H).$ In particular, if $\omega$ is a tracial state on $C^*_r(G)$ such that $\omega(\lambda_s)=0$ when $s\not\in H$, then $\omega\circ E=\omega.$
\end{lemma}

\begin{lemma} \label{lem:repcondexp} Let $H\leq G$ be amenable discrete groups.  Let $\tau$ be a trace on $G$ that vanishes on $G\setminus H.$ Let  $E:C^*(G)\rightarrow C^*(H)$ be the conditional expectation from Lemma \ref{lem:condexp}.
The map $E_\tau:C_{\pi_\tau}^*(G)\rightarrow C^*_{\pi_\tau}(H)$ given by $E_\tau(\pi_\tau(x))=\pi_\tau(E(x))$ extends to a well-defined $\tau$-preserving conditional expectation onto $C^*_{\pi_\tau}(H).$ 
\end{lemma}
\begin{proof}

Since ${\tau}$ is a tracial state, for any $x \in C^*(G)$, we have $\pi_\tau(x)=0$ if and only if ${\tau}(x^*x)=0.$
Consequently, $\pi_\tau(x)=0$ implies ${\tau}(E(x)^*E(x))\leq \tau(E(x^*x))=\tau(x^*x)=0,$ so $\pi_\tau(E(x))=0.$  Therefore $E_\tau$ is well-defined and $\tau$-preserving by Lemma \ref{lem:condexp}.
\\\\
The map $E_\tau$ is clearly idempotent so we only need to check it is contractive. This proceeds as in the case of building conditional expectations in the arena of finite von Neumann algebras. We include the proof for the convenience of the reader.
\\\\
For each $x\in C^*(G)$, let $x_\tau\in H_\tau$ be the canonical image of $x.$ Since $E$ is $\tau$-preserving, the map $P(x_\tau):=E(x)_\tau$ extends to an orthogonal projection on $H_\tau.$  Since $E(x)\in C^*(H)$,
\begin{align*}
\V  \pi_\tau(E(x)) \V&=\sup_{ y,z\in C^*(H), {\tau}(y^*y)={\tau}(z^*z)=1}|\la \pi_\tau(E(x))y_\tau,z_\tau  \ra|\\
&=\sup_{ y,z\in C^*(H),{\tau}(y^*y)={\tau}(z^*z)=1}|{\tau}(z^*E(x)y)|\\
&=\sup_{y,z\in C^*(H), {\tau}(y^*y)={\tau}(z^*z)=1}|\la P(x_\tau),(zy^*)_\tau  \ra|\\
&=\sup_{y,z\in C^*(H), {\tau}(y^*y)={\tau}(z^*z)=1}|\la x_\tau,P((zy^*)_\tau)  \ra|\\
&=\sup_{y,z\in C^*(H), {\tau}(y^*y)={\tau}(z^*z)=1}|\la x_\tau,(zy^*)_\tau  \ra|\\
&\leq\sup_{y,z\in C^*(G), {\tau}(y^*y)={\tau}(z^*z)=1}|\la x_\tau,(zy^*)_\tau  \ra|\\
&=\V \pi_\tau(x) \V.
\end{align*}
\end{proof}

\begin{lemma} Let $G$ be a finitely generated  nilpotent group  and  $\tau$  an  extreme trace on $G$ such that $\tau(g)=1$ implies $g=e.$ 
Let $N\leq Z(G)$ be a finite index subgroup of $Z(G)$. % such that $\tau$ is faithful on $N$.  
 Let $\omega$ be the trivial extension of $\tau|_N$ to $G.$  Then there is a central projection $p\in C^*_{\pi_\omega}(G_f)\cap C^*_{\pi_\omega}(G)'$ such that 
\[C^*_{\pi_\tau}(G)\cong pC^*_{\pi_\omega}(G).\]
\label{lemma:cutdown}
\end{lemma}
\begin{proof} Our hypotheses, combined with 
\cite[Theorem 4.5]{Carey84},
 imply that $\tau$ vanishes on  $G\setminus G_f.$ 
 By Lemma \ref{lem:weakcontainment}, the representation $\pi_\omega$ weakly contains $\pi_\tau.$ Therefore the *-homomorphism $\sigma:C^*_{\pi_\omega}(G)\rightarrow C^*_{\pi_\tau}(G)$ given by $\sigma(\pi_\omega(x))=\pi_\tau(x)$ is well-defined. From Lemma \ref{lem:repcondexp} (with $H=G_f$) we have trace-preserving conditional expectations $E_\omega:C^*_{\pi_\omega}(G)\rightarrow C^*_{\pi_\omega}(G_f)$ and $E_\tau:C^*_{\pi_\tau}(G)\rightarrow C^*_{\pi_\tau}(G_f).$  By the definitions of these expectations in Lemma \ref{lem:repcondexp} we have, for any   $x\in C^*(G)$,
\begin{equation*}
E_\tau(\sigma(\pi_\omega(x)))=E_\tau(\pi_\tau(x))=\pi_\tau(E(x))=\sigma(\pi_\omega(E(x)))=\sigma(E_\omega(\pi_\omega(x))). 
\end{equation*}
Equivalently, the following diagram commutes:
\begin{equation} \label{eq:commute}
\xymatrix{
C^*_{\pi_\omega}(G) \ar[rr]^\sigma \ar[d]^{E_\omega}&& C^*_{\pi_\tau}(G)\ar[d]^{E_\tau}\\
C^*_{\pi_\omega}(G_f)  \ar[rr]^\sigma  && C^*_{\pi_\tau}(G_f)}
\end{equation}
By \cite[Proposition 5.1]{Carey84} there is an irreducible representation $\pi$ of $G$ such that $\pi(x)\mapsto \pi_\tau(x)$ defines an isomorphism from $C^*_\pi(G)$ onto $C^*_{\pi_\tau}(G).$  
The irreducibility of $\pi$ then implies that $\pi_\tau(z)\in \C\cdot 1_{H_\tau}$ for all $z\in Z(G)$. 

Using formula \eqref{eq:twist} we see that $\pi_\omega(n) s_\omega = \omega(n) s_\omega $ for all $s \in c(G/N)$, hence $\pi_\omega(n) = \omega(n) \cdot 1_{H_\omega}$.   
By Lemma \ref{lem:findex}, the group $N$ has finite index in $G_f.$ 
Therefore,  $C^*_{\pi_\omega}(G_f)$ is finite dimensional by Lemma \ref{lem:twistedgroup}.

Since $C^*_{\pi_\omega}(G_f)$ is finite dimensional there is a projection $p\in C^*_{\pi_\omega}(G_f)\cap C^*_{\pi_\omega}(G_f)'$ such that 
\begin{equation}
\textup{ker}(\sigma)\cap C^*_{\pi_\omega}(G_f)=(1-p)C^*_{\pi_\omega}(G_f). \label{eq:kernel}
\end{equation}
Let $x\in C^*_{\pi_\omega}(G).$  By Lemma \ref{lem:repcondexp}, the conditional expectations $E_\omega$ and $E_\tau$ preserve their respective traces. Using this fact we obtain
\begin{align*}
x \in \ker(\sigma) & \Longleftrightarrow \tau(\sigma(x^*x))=0 \\
& \Longleftrightarrow \tau(E_\tau(\sigma(x^*x)))=0\\
& \Longleftrightarrow E_\tau(\sigma(x^*x))=0\\
& \Longleftrightarrow\sigma(E_\omega(x^*x))=0  \quad \textup{By (\ref{eq:commute})} \\
& \Longleftrightarrow pE_\omega(x^*x)=0 \quad \textup{By (\ref{eq:kernel})}\\
& \Longleftrightarrow E_\omega(px^*xp)=0\\
& \Longleftrightarrow \omega(E_\omega(px^*xp))=0\\
& \Longleftrightarrow \omega(px^*xp)=0\\
& \Longleftrightarrow xp=0\\ 
& \Longleftrightarrow x\in C^*_{\pi_\omega}(G)(1-p).
\end{align*}
It follows that $(1-p)$ is central.  Indeed, for any $x \in C^*_{\pi_\omega}(G),$ we have $(1-p)x=(x^*(1-p))^*\in C^*_{\pi_\omega}(G)(1-p)$ (it follows from the above list of equivalences that $C^*_{\pi_\omega}(G)(1-p)$ is self-adjoint). Then  $(1-p)x=y(1-p)$ for some $y\in C^*_{\pi_\omega}(G).$  Consequently, $(1-p)x(1-p)=y(1-p)=(1-p)x.$  It follows that $(1-p)xp=0.$ The same argument applied to $x^*$ shows $px(1-p)=0$, in other words,  that $x$ commutes with $p.$ 
\end{proof}
\begin{theorem} \label{thm:structure} Let $G$ be a finitely generated nilpotent group and $\pi$ a faithful irreducible representation of $G.$ 
There is a torsion free, finite index subgroup $N\leq Z(G)$, a 2-cocycle $\sigma$ on $G/N$,
  and a central projection $p\in C^*(G/N,\sigma)$, such that $C^*_\pi(G)\cong pC^*(G/N,\sigma).$ Moreover, $\sigma$ is homotopic to the trivial cocycle.
  \end{theorem}
  \begin{proof} As in the proof of Theorem \ref{thm:main} there is an extreme trace $\tau$ on $G$ satisfying the hypotheses of Lemma \ref{lemma:cutdown} such that $C^*_\pi(G)\cong C^*_{\pi_\tau}(G).$
  
 Since $Z(G)$ is a finitely generated abelian group it contains a finite index torsion free subgroup $N.$  Then $\widehat{N}$ is a $d$-torus and in particular path-connected.  It follows that the cocycle defined in (\ref{eq:cocycle}) for any $\omega\in \widehat{N}$ is homotopic to the trivial cocycle.  The conclusion now follows from Lemmas \ref{lem:twistedgroup} and \ref{lemma:cutdown}.   
 \end{proof}
 \begin{remark} Let $C^*(G/N,\sigma)$ be as in Theorem \ref{thm:structure}. Since $G/N$ is amenable,  \cite[Theorem 1.9]{Echterhoff10} applies to $C^*(G/N,\sigma).$  In particular, since $\sigma$ is homotopic to the trivial cocycle, we have $K_*(C^*(G/N,\sigma))\cong K_*(C^*(G/N)).$  Since K-theory is additive over direct sums, we feel that Theorem \ref{thm:structure} may be beneficial for K-theory calculations of C*-algebras generated by irreducible representations of nilpotent groups (see for example \cite{Eckhardt14a} for this idea in action).
 \end{remark}
\bibliographystyle{plain}

\end{document}